\documentclass[12pt,oneside]{amsart}
\usepackage[margin=1in]{geometry}

\usepackage{graphicx}
\usepackage{amssymb}
\usepackage{tikz}
\usepackage{amsmath,amsthm,amscd,amssymb}
\usepackage{latexsym}
\usepackage[colorlinks,citecolor=red,pagebackref,hypertexnames=false]{hyperref}
\hypersetup{%
  colorlinks = true,
  citecolor=red,
  linkcolor  = black
}

\usepackage{ulem}
\usepackage{soul}
\numberwithin{equation}{section}
\theoremstyle{plain}
\newtheorem{theorem}{Theorem}[section]

\newtheorem{corollary}[theorem]{Corollary}
\newtheorem{proposition}[theorem]{Proposition}
\newtheorem{conjecture}[theorem]{Conjecture}

\theoremstyle{definition}
\newtheorem{definition}[theorem]{Definition}

\theoremstyle{remark}
\newtheorem{remark}[theorem]{Remark}

\newtheorem{case[theorem]}{Case}

\date{\today}      
\author{A. Iosevich and A. Mayeli} 
\address{Department of Mathematics, University of Rochester, Rochester, NY}
\email{iosevich@gmail.com}
\address{Department of Mathematics, CUNY Graduate Center, New York, NY}
\email{amayeli@gc.cuny.edu } 
\thanks{A.I. was supported in part by the National Science Foundation under grant no. HDR TRIPODS - 1934962 and by the NSF DMS - 2154232. A.M. was supported in part by AMS-Simons Research Enhancement Grant and the PSC-CUNY research grants.}

\begin{document}

\title[Signal recovery and restriction]{ Uncertainty Principles on Finite Abelian Groups, Restriction Theory, and Applications to sparse signal recovery}

\begin{abstract} 
Let $G$ be a finite abelian group. Let  $f: G \to {\mathbb C}$ be a signal (i.e. function). The classical uncertainty principle asserts   that the product of the size  of the support of $f$ and its Fourier transform $\hat f$, $\text{supp}(f)$ and $\text{supp}(\hat f)$ respectively, must satisfy the condition: 
$$|\text{supp}(f)| \cdot |\text{supp}(\hat f)| \geq |G|.$$ 
In the first part of this paper, we improve the uncertainty principle for signals with Fourier transform supported on generic sets. This improvement is achieved by employing  {\it the restriction theory} and {\it the  Salem set}  mechanism from harmonic analysis. Then we investigate some applications of uncertainty principles that were developed in the first part of this paper, to the problem of unique recovery of finite sparse signals in the absence of some frequencies.
 
Donoho and Stark (\cite{DS89})  showed that a signal of length $N$ can be recovered exactly, even if some of the frequencies are unobserved, provided that the product of the size of the number of non-zero entries of the signal and the number of missing frequencies is not too large, leveraging the classical uncertainty principle for vectors. Our results broaden the scope for a natural class of signals in higher-dimensional spaces. In the case when the signal is binary, we provide a very simple exact recovery mechanism through the DRA algorithm. 

 
\end{abstract}  

\maketitle

\tableofcontents
\addtocontents{toc}{\protect\setcounter{tocdepth}{1}}
\section{Introduction}

The purpose of this paper is to examine some basic questions in the realm of signal recovery from incomplete data in signal processing, from the point of view of Fourier uncertainty principles obtained using the restriction theory for the Fourier transform. The questions are motivated by the seminal paper by Donoho and Stark (\cite{DS89})   where the uncertainty principle was used in a fundamental way to affect the exact recovery of a sequence encoded in terms of its Discrete Fourier Transform  (DFT). 

\vskip.125in 

The main thrust of this work is to investigate how classical restriction theory which has played such an important role in modern harmonic analysis comes into play in exact signal recovery via suitable uncertainty principle estimates. We also develop conditions under which exact signal recovery can be accomplished very simply and efficiently. Finally, we develop a simple procedure that allows us to both discretize a signal and perform an efficient recovery procedure. 

\vskip.125in 

This article is organized as follows. In  Section \ref{donohostarksubsection} we describe the Donoho-Stark approach to exact signal recovery via the classical uncertainty principle. Section \ref{upsection} is dedicated to the exposition of a variety of uncertainty principles, using restriction theory, randomness, and decay properties of the Fourier transform. The interaction between the parameters associated with these quantities is discussed as well. In Section \ref{ersection}, we describe the application of the uncertainty principles in Section \ref{upsection} to exact signal recovery. We also describe how these ideas combine with an elementary approach to exact signal recovery we call DRA (see Definition \ref{dradefinition} below), the direct rounding algorithm. In Section \ref{grsection}, we discuss the exact recovery problem in a general setting, with a particular focus on the celebrated Euclidean restriction conjecture. Finally, the remaining proofs are given in Section \ref{proofsection}. 

\section{Preliminaries}

\subsection{Donoho-Stark, support size, and the uncertainty principle.} \label{donohostarksubsection}
 In order to introduce our viewpoint, we need to establish some notation regarding the discrete Fourier transform on finite abelian groups. 
For clarity, we narrow our attention to the finite groups over cyclic groups, i.e.,  $G={\mathbb Z}_N^d$,  where ${\mathbb Z}_N$ is the cyclic group (mod) $N$. 
For a given signal (i.e. function)  $f: {\mathbb Z}_N^d \to {\mathbb C}$, the  Fourier transform $\hat f: {\mathbb Z}_N^d \to {\mathbb C}$ is  a function defined by 
$$ \hat{f}(m)=N^{-d} \sum_{x \in {\mathbb Z}_N^d} \chi(-x \cdot m) f(x),$$ where $\chi(t)=e^{\frac{2 \pi i t}{N}}, \ t \in {\mathbb Z}_N$. Here $m$ is the element of the dual group $\widehat{{\mathbb Z}_N^d}$, that is identified with ${\mathbb Z}_N^d$ itself, and $x\cdot y$  is the dot product in $\Bbb Z_N^d$.  
 The Fourier inversion formula is given by 
 \begin{align}\label{FIN}f(x)=\sum_{m \in {\mathbb Z}_N^d} \chi(x \cdot m) \widehat{f}(m). 
 \end{align}

 The Plancherel identify is given by 
 \begin{align}\label{PI}
 \sum_{x\in \Bbb Z_N^d} |f(x)|^2 = N^d \sum_{m\in \Bbb Z_N^d} |\hat f(m)|^2
 \end{align}
(For a  description of the fundamentals of Fourier analysis on groups, see e.g. \cite{Babai2002,R62,Terras99}.)

The  {\it classical discrete-time uncertainty principle}  for the cyclic groups $\Bbb Z_N$ is due to Donoho and Stark (\cite{DS89}),  and for any finite abelian group $G$ is due to Smith (\cite{Smith90}). For a sharper uncertainty principle result for $\Bbb Z_N$, $N$ a prime, see Tao's result \cite{TaoUncertainty}. 

The principle for the group $G=\Bbb Z_N^d$  asserts that  
$f: {\mathbb Z}_N^d\to {\mathbb C}$ is a non-zero function with support $\text{supp}(f)$ and $\hat f:  {\mathbb Z}_N^d \to {\mathbb C}$ denotes the Fourier transform with support $\text{supp}(\hat f)$, 
then 
 \begin{equation} \label{UPequation-higher}  | \text{supp}(f)| \cdot | \text{supp}(\hat f)| \ge N^d. \end{equation} 

This bounds the time-bandwidth product  from below. 
This principle can be expressed in $1$-dimensional case ${\mathbb Z}_N$ as follows: 
 Let  $(x_i)_{i=1}^{N-1}$ be a finite vector, and let the corresponding discrete Fourier transform $(\hat x_w)_{w=1}^{N-1}$ obtained through the DFT. If  the original sequence has $N_t$ non-zero entries and the transformed sequence has $N_\xi$ non-zero entries, then 
  \begin{equation} \label{UPequation} N_t \cdot N_\xi \ge N. \end{equation} 

Using this uncertainty principle in one dimension, Donoho and Stark established the following result for the recovery of finite one-dimensional signals in the presence of no noise.  

\begin{theorem}[\cite{DS89}]\label{DS} Let $f: \Bbb Z_N\to \Bbb C$ be a finite signal of length $N$ in $\mathbb Z_N$ with   $N_t$ non-zero entries. Suppose that the set of unobserved frequencies $\{\hat f(m)\}_{m\in \mathbb Z_N}$ is of size $N_w$. Then the signal $f$ can be   `recovered uniquely' from the observed frequencies if 
\begin{align} \label{sufficient-size-for-recovery} N_t \cdot N_w < \frac{N}{2}.
\end{align} \end{theorem}

\begin{figure}
\label{arraydftpicture}
\centering
\includegraphics[scale=.4]{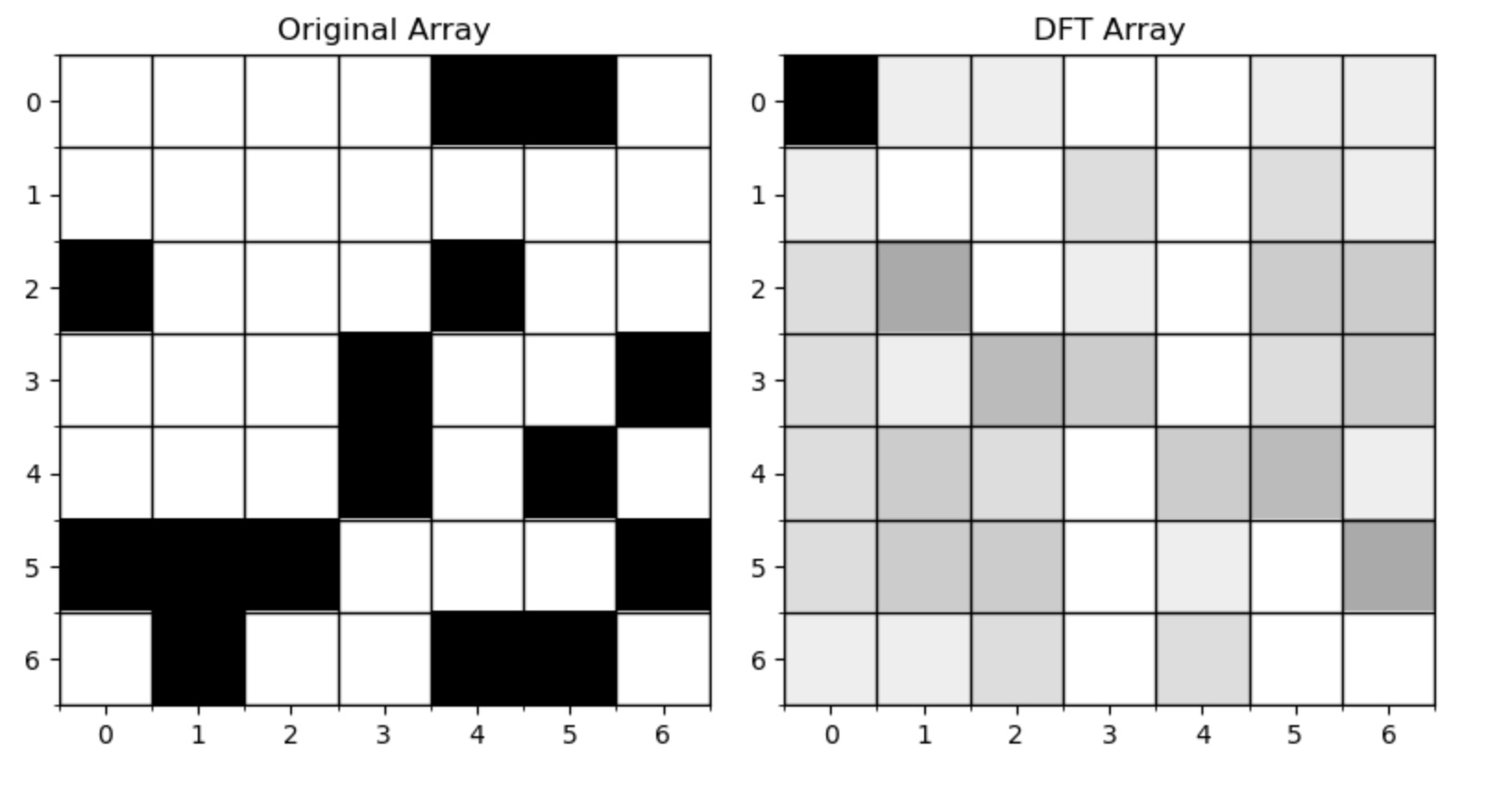}
 \caption{\tiny{(left) A sparse $7\times 7$ array of $1$'s and $0$'s in ${\mathbb Z}_7^2$, and  (right) the array of its discrete Fourier transform on ${\mathbb Z}_7^2$ shaded according to the magnitude of the Fourier coefficients.}}
\end{figure}

The result states that for successful unique recovery,  the signal must exhibit some degree of sparsity,  and a limited number of frequencies can be absent. 

The recovery problem falls within the realm of inverse problems, and it has wide-ranging applications in computer vision, cryptography, data analysis, digital logic circuits, and many other areas of computer science, data science, mathematics, and engineering. See, for example, \cite{AMS99,BGIKS08,CR05,CRT06,RV06, RV08, stallings2006cryptography,schneier2007applied,gonzalez2018digital,ciletti2007digital,janert2010data} and the references contained therein. 

\subsubsection{From uncertainty principle to exact recovery} \label{subsubsectionuptoer} Donoho and Stark used an optimization and $\ell^2$-minimization  technique to recover the spare signal in the presence of missing frequencies, while they used the uncertainty principle to prove the uniqueness of the recovery. The proof of the uniqueness goes as follows:   Let  $f\neq 0$  with support $E=\text{supp}(f)$, and let $S\subset {\mathbb Z}_N$ be the set where the corresponding frequencies $\hat f(m)$, $m\in S$, are absent. 
Assume that $r$ and $g$ are two signals recovering $f$  with  $\hat r(m) = \hat f(m) = \hat g(m)$ for all $m\not\in S$. Then $(\hat r-\hat g)(m)= 0$ for all $m\not\in S$. Define $h=r-g$. Thus, $\text{supp}(\hat h) \subset S$. On the other side, we have  
$\text{supp}(r) = \text{supp}(g) = \text{supp}(f)$.   
%
 %
This implies that  $h$ is supported on a set of size at most $2|E|$, while $\widehat{h}$ is supported on $S$. By the uncertainty principle (\ref{UPequation}), no such $h$ can exist if (\ref{sufficient-size-for-recovery}) holds, and the proof of the uniqueness is complete.  The proof of uniqueness remains identical in higher dimensions. 

Traditional approaches to establishing the uncertainty principle involve advanced techniques like Weyl's inequality or the use of prolate spheroidal wave functions.  (For information on  prolate spheroidal wave functions,  see e.g. \cite{Bell1, ArieAzita23}.) These methods delve into complex concepts such as eigenfunctions of the Fourier transform (as demonstrated by Weyl in \cite{weyl1928}) and eigenfunctions of compact operators, as introduced by Landau and Pollak in \cite{Bell2}. In the discrete setting, however, a much simpler approach can be employed, and this point of view is going to lead us to {\it an improved version}  of the discrete-time uncertainty principle under some natural conditions.  

For the remainder of the paper, we will omit mentioning discrete-time and simply refer to the uncertainty principle, provided it is clear from the context. 

\section{Sharper uncertainty principles in  $\Bbb Z_N^d$} \label{upsection}
\subsection{Via the 
restriction theory}\label{restricionsubsection}

The goal of this section is to prove that the uncertainty principle can 
further {\it refined} using restriction theory. This result is stated in Theorem \ref{mainUPwithRT}.  For the sake of self-containment, we shall 
illustrate how the uncertainty principle in finite settings can be derived through the inverse Fourier transform.

{\it Proof of the uncertainty principle for finite abelian groups:} We will illustrate the proof for $G=\Bbb Z_N^d$ since any finite abelian group can be expressed as a finite product of cyclic groups, and the proof remains the same. 

Suppose that $f$ is supported in a set $E$, and $\widehat{f}$ is supported in a set $\Sigma$. Then by the Fourier Inversion Formula \eqref{FIN}, we have  
$$ f(x)=
\sum_{m \in S} \chi(x \cdot m) \widehat{f}(m).$$

By applying   the Cauchy-Schwarz inequality, we can derive the following: For any $z\in \Bbb Z_N^d$   
\begin{align}\notag  {|f(z)|}^2 \leq |\Sigma| \cdot \sum_{m \in \Sigma} {|\widehat{f}(m)|}^2 
  & = |\Sigma| \cdot \sum_{m \in {\mathbb Z}_N^d} {|\widehat{f}(m)|}^2 \\\label{trivialestimateequation}  
 &= |\Sigma| \cdot N^{-d} \cdot \sum_{x \in {\mathbb Z}_N^d} {|f(x)|}^2  \\\label{eq:x}
&= |\Sigma| \cdot N^{-d} \cdot \sum_{x \in E} {|f(x)|}^2, 
\end{align} 
 
where  in \eqref{trivialestimateequation} we used the  Plancherel identity \eqref{PI}, and in \eqref{eq:x} we used the fact that $f$ is supported in $E$. 
Summing both sides over $z\in E$ and dividing both sides by $\sum_{x \in E} {|f(x)|}^2$, we see that 
\begin{equation} \label{UPdequation} |E| \cdot  |\Sigma| \ge N^d, \end{equation} recovering (\ref{UPequation}) in the case $d=1$, and \eqref{UPequation-higher} in higher dimension $d$. 

\begin{remark} \label{manbitesdogremark} It is important to note that the bound (\ref{UPdequation}) is essentially sharp. For example, suppose that $N$ is an integer,  and $E$ is a $k$-dimensional subspace of ${\mathbb Z}_N^d$. Then by a direct calculation, $\widehat{E}(m)=N^{-(d-k)} E^{\perp}(m)$, where $E^{\perp}$ is the orthogonal subspace to the space $E$, i.e., 
$e^{\frac{2\pi i t\cdot n}{N}}=1$ for all $t\in E$ and $n\in E^\perp$.  When $m=\vec{0}$,    this implies that    $|E| \cdot |E^{\perp}|=N^d$. This indicates that if the frequencies in $E^{\perp}$ are missing or unobserved, it hinders the recovery of the original information.  
However, such examples are very rare. In fact, one can show that these are the only examples where the equality in (\ref{UPdequation}) holds. 

\end{remark} 

\vskip.125in 

\begin{remark} It is interesting to note that if $N$ is prime, the problem takes on a variety of interesting additional features. For example, if $d=1$, the classical uncertainty principle can be replaced by a stronger version proved by Tao (\cite{TaoUncertainty}). If $N=2$, it is known (\cite{IMP17}, \cite{HIPRV18}) that if $E \subset {\mathbb Z}_N^2$ and $\widehat{E}(m)=0$ for some $m \in {\mathbb Z}_N^2$, then $E$ has the same number of points on all lines perpendicular to $m$. This suggests a potentially interesting link between the exact recovery questions and tiling problems in vector spaces over finite fields. See, for example, \cite{FKS22}, \cite{FMV19}, and the references contained therein. 

\end{remark} 

\vskip.125in 

The key point we are going to exploit is that if $\Sigma$, the support of $\hat f$,  is a typical set, then instead of using the support-driven identity   
\begin{equation} \label{trivialestimateequationpure} \sum_{m \in \Sigma} {|\widehat{f}(m)|}^2 = \sum_{m \in {\mathbb Z}_N^d} {|\widehat{f}(m)|}^2 \end{equation} used in \eqref{trivialestimateequation} above in the derivation of (\ref{UPdequation}), followed by estimating the $L^2$ norm of $f$ over its support, we can bound the left-hand side of (\ref{trivialestimateequationpure}) by a suitably scaled $L^p$-norm of $f$ for some $p<2$, resulting in a generally better uncertainty principle. This may seem counter-intuitive since (\ref{trivialestimateequationpure}) is an identity owing to the support assumption on $\widehat{f}$. The gain comes from comparing $L^p$ norms, which leads to a lesser strain on the support of the signal $f$. In order to execute this idea, we bring in the following notion from classical restriction theory. 

\begin{definition} \label{restrictiondef} 
Let $S \subset {\mathbb Z}_N^d$. We say that a $(p,q)$-restriction estimation ($1 \leq p \leq q\leq \infty$) holds for $S$ if there exists a uniform constant  $C_{p,q}$ (independent of $N$ and $S$) such that  for any function $f:\Bbb Z_N^d\to \Bbb C$
\begin{equation} \label{restrictionequation} {\left( \frac{1}{|S|} \sum_{m \in S} {|\widehat{f}(m)|}^q \right)}^{\frac{1}{q}} \leq C_{p,q} N^{-d} {\left( \sum_{x \in {\mathbb Z}_N^d} {|f(x)|}^p \right)}^{\frac{1}{p}}. \end{equation} 
\end{definition} 
When $q=\infty$,  the left-hand side is replaced by the supremum norm, i.e.,  $\|\hat f \|_{\infty}=\max\{|\hat f(m)|\}_{m\in S}$.  

\vskip.125in 

The definition indicates that when $N$ is large, the frequency concentration on the set $\Sigma$ is relatively low. This characteristic proves to be quite advantageous when it comes to signal recovery, especially in scenarios where frequencies outside of $S$ are missing. See Corollary \ref{ERwithRT}  below. The case $q=\infty$ immediately points to the relationship between the decay properties of the Fourier transform of $S$ and the restriction phenomenon. We shall explore this phenomenon in more detail in Subsection \ref{arbitrarysizesubsection}. 

A tremendous amount of work has been done on the restriction phenomenon in vector spaces over finite fields and modules over rings. See for example, \cite{HW18,IK10,IK10b,IKL17,MT04} and the references contained therein. These results mostly deal with restriction to spheres, paraboloids, and other algebraic surfaces in finite settings. The Euclidean restriction theory is discussed briefly in Section \ref{grsection} below. While these situations are interesting in the context of signal recovery, the most interesting case is where the restriction surface is random, and we are going to develop this theory later in this paper.  

\vskip.125in 

Our first result is the following. 

\begin{theorem}[Uncertainty Principle via Restriction Estimation]\label{mainUPwithRT}  Suppose that $f: {\mathbb Z}_N^d\to \Bbb C$ is supported in $E \subset {\mathbb Z}_N^d$, and $\hat{f}:\Bbb Z_N^d\to \Bbb C$ is supported in $\Sigma \subset {\mathbb Z}_N^d$. Suppose that the restriction estimation  (\ref{restrictionequation}) holds for $\Sigma$ for a pair $(p,q)$, $1\leq p\leq q$.  Then 
\begin{equation} \label{UPwithRTequation} {|E|}^{\frac{1}{p}} \cdot |\Sigma| \ge \frac{N^d}{C_{p,q}}. \end{equation} 
\end{theorem} 

\begin{remark} It is interesting to note that  for the pair $(p,q)=(1,2)$, the  restriction estimation holds for any set $S$ with constant $C_{1,2}=1$, recovering the classical uncertainty principle (\ref{UPdequation}). \end{remark} 

\vskip.125in

This raises the question under which conditions a non-trivial restriction estimation can hold for a given $\Sigma \subset {\mathbb Z}_N^d$. A sample result is the following. 

\begin{theorem} \label{restrictionlambda4} Let $\Sigma \subset {\mathbb Z}_N^d$ with the property that 
\begin{equation} \label{size} |\Sigma|= \Lambda_{\text{size}} N^{\frac{d}{2}}, \end{equation} 
and 
\begin{equation} \label{lambda4condition} |\{(x,y,x',y') \in U^4:\  x+y=x'+y' \}| \leq \Lambda_{\text{energy}} \cdot {|U|}^2 \end{equation} for every $U \subset \Sigma$.

\vskip.125in 

Then the restriction estimation holds for $\Sigma$ for $(p,q)$, where $p=4/3$ and $q=2$.   Indeed,  for any $f: {\mathbb Z}_N^d \to {\mathbb C}$,
\begin{equation} \label{restrictionlambda4estimate} {\left( \frac{1}{|\Sigma|} \sum_{m \in \Sigma} {|\widehat{f}(m)|}^2 \right)}^{\frac{1}{2}} \leq \Lambda_{\text{size}}^{-\frac{1}{2}} \cdot \Lambda_{\text{energy}}^{\frac{1}{4}} \cdot N^{-d} {\left( \sum_{x \in {\mathbb Z}_N^d} {|f(x)|}^{\frac{4}{3}} \right)}^{\frac{3}{4}}. \end{equation} 

\end{theorem}

\vskip.125in 

\begin{remark} It is interesting to note that the assumption (\ref{lambda4condition}) holds in a variety of natural situations. For example, if $d=2$ $N$ is an odd prime, and 
$$\Sigma=\{x \in {\mathbb Z}_N^2: x_1^2+x_2^2=1\},$$ the unit circle, then (\ref{lambda4condition}) is satisfied with $\Lambda_{energy}=3$ and $\Lambda_{size}$ essentially equal to $1$. The resulting restriction theorem was first established by the first listed author and Doowon Koh in \cite{IK08}. See also \cite{HW18} for a variety of restriction theorems over ${\mathbb Z}_N^d$. 
\end{remark} 

\begin{remark} \label{randomenergytheorem} Let $\Sigma\subset {\mathbb Z}_N^d$ of size $|\Sigma|=\Lambda_{\text{size}} N^{\frac{d}{2}}> N^{\frac{d}{2}}$. Suppose that $\Sigma$ is chosen randomly with respect to the uniform distribution. Then for every $U \subset \Sigma$, the expected value of 
$$ |\{(x,y,x',y') \in U^4: x+y=x'+y'\}|$$ is bounded by 
\begin{equation} \label{energysizerandomequation} (5+c^2){|U|}^2. \end{equation}  

The proof, which uses Chernoff's classical bound shows considerable concentration around the mean and shows that with very high probability, the desired energy inequality holds. The result follows easily from the calculations in \cite{DSSS13}. 
\end{remark} 

\vskip.125in 

\begin{remark} We note that Theorem \ref{restrictionlambda4} is just one example of the relationship between additive energy and restriction. It is not difficult to show that if 
$$ |\{(x^1, \dots, x^k, y^1, \dots, y^k) \in U^{2k}: x^1+x^2+\dots+x^k=y^1+y^2+\dots+y^k\}| \leq \Lambda_{\text{energy}} {|U|}^k$$ for every $U \subset S$, then we obtain the restriction estimate with the exponents $\left(\frac{2k}{2k-1},2 \right)$ with the uniform constant suitably dependent on $\Lambda_{\text{energy}}$ and $\Lambda_{\text{size}}$. More work is required to obtain an appropriate variant of (\ref{energysizerandomequation}). This investigation will be conducted in the sequel. 
\end{remark}  

\vskip.125in 

\begin{remark} Throughout this paper we are going to stick to the pure support conditions, namely, the signal is supported in a set $E$, and its Fourier transform is supported in a set $S$. In practice, many of the arguments go through, up to a constant, if we assume that $f$ is concentrated in $E$ in a suitable sense. For example, we could assume that 
$$ {\left( \sum_{x \in {\mathbb Z}_N^d} {|f(x)|}^p \right)}^{\frac{1}{p}} \leq C_{p,E} {\left( \sum_{x \in E} {|f(x)|}^p \right)}^{\frac{1}{p}},$$ which would allow all of our results to go through at the cost of the constant $C_{p,E}$. This and related notions will be systematically explored in the sequel. 
\end{remark} 

\vskip.125in 

\begin{remark} The $(\frac{4}{3},2)$ restriction theorem in Theorem \ref{restrictionlambda4} extends, by interpolation, to a $(p,2)$ restriction theorem for any $1 \leq p \leq \frac{4}{3}$, since the $(1,2)$ restriction theorem always holds as we noted above. \end{remark} 

\vskip.125in 

\subsection{Salem sets and Salem uncertainty principle}
\label{arbitrarysizesubsection}
We are now going to explore uncertainty principles based on the assumption that the underlying sets are Salem sets, named after Raphael Salem, the mathematician who first discovered them and studied their properties (see e.g. \cite{S50}). In the finite setting, the definition requires a bit of care. 




\begin{definition}[Salem sets]\label{SalemSet} A set $S\subset \Bbb Z_N^d$ is a Salem set at level ${\Lambda}_{Salem}$ if  \begin{equation} \label{fouriersizeencoderequation} |\widehat{S}(z)| \leq \Lambda_{\text{Salem}} \cdot N^{-d} \cdot {|S|}^{\frac{1}{2}} \quad \forall \ z\neq 0. \end{equation} 
\end{definition}  
Notice that every set is a Salem set with the constant $\Lambda_{\text{Salem}}= {|S|}^{\frac{1}{2}}$. This follows from the following simple argument: 
By the definition of the Fourier transform, the inequality 
$$ |\widehat{S}(z)| \leq N^{-d} |S|$$ always holds. Therefore,  we can write 
$$ |\widehat{S}(z)| \leq {|S|}^{\frac{1}{2}} \cdot N^{-d} {|S|}^{\frac{1}{2}}.$$ 

In general, however, the estimate on $\Lambda_{\text{Salem}}$ is much better, as we will show later in  Proposition \ref{randomisgood}. Indeed, it shows that with probability $1-N^{-d \epsilon}$, a randomly chosen set $S$ satisfies the bound $|\widehat{S}(z)| \leq \Lambda_{\text{Salem}} N^{-d} {|S|}^{\frac{1}{2}}$, for $z \not=(0, \dots, 0)$, with $\Lambda_{\text{Salem}} \leq \sqrt{(1+\epsilon)d\log(n)}$.

\vskip.1in 



\vskip.125in 

\begin{theorem} \label{theoremsalem} (Salem Uncertainty Principle) Let $E \subset {\mathbb Z}_N^d$.
Suppose   that $\Sigma$ is Salem at level $\Lambda_{\text{Salem}}$.  
Then for any function $f:\Bbb Z_N^d\to \Bbb C$ with support in $E$ and the Fourier transform $\hat f$ with support in $\Sigma$, we have 
\begin{equation} \label{salemmainestimate} |E| \cdot {|\Sigma|}^{\frac{3}{4}} \ge N^d \cdot \sqrt{\frac{1-dens(\Sigma)}{\Lambda_{\text{Salem}}}},
\end{equation} where $dens(\Sigma)=N^{-d}|\Sigma|$ is  the {\it density} of the set $\Sigma$. 
\end{theorem} 

\vskip.125in 

\begin{remark} This estimate complements the results we obtained using Theorem \ref{mainUPwithRT}, Theorem \ref{restrictionlambda4} and Remark \ref{randomenergytheorem} because (\ref{salemmainestimate}) allows for a non-trivial result even in the case when $\Sigma={\mathbb Z}_N^d$, as long as $|E|$ is suitably small. 

Also, one can combine Theorem \ref{theoremsalem} and Theorem \ref{mainUPwithRT} producing a hybrid lower bound where the powers of both $|E|$ and $|S|$ are smaller than $1$. 
\end{remark} 

\vskip.125in 

\begin{remark} The interested reader can check that the proof of Theorem \ref{theoremsalem} given below also yields a restriction theorem. In the course of proving Theorem \ref{theoremsalem}, we show that 
\begin{equation} \label{l1bound} {\left( \frac{1}{|\Sigma|} \sum_{m \in \Sigma} {|\widehat{f}(m)|}^2 \right)}^{\frac{1}{2}} \leq 
\frac{N^{-d} \cdot {|\Sigma|}^{-\frac{1}{4}} \cdot \Lambda^{\frac{1}{2}}_{\text{Salem}} \cdot 
\sum_x |f(x)|}{\sqrt{1-dens(S)}}. \end{equation} 

We can also check using Plancherel that 
\begin{equation} \label{l2bound} {\left( \frac{1}{|\Sigma|} \sum_{m \in \Sigma} {|\widehat{f}(m)|}^2 \right)}^{\frac{1}{2}} \leq {|\Sigma|}^{-\frac{1}{2}} N^{-\frac{d}{2}} \cdot {\left( \sum_x {|f(x)|}^2 \right)}^{\frac{1}{2}} \end{equation} 
$$ =N^{-d} \cdot {\left( \frac{N^d}{|\Sigma|} \right)}^{\frac{1}{2}} {\left( \sum_x {|f(x)|}^2 \right)}^{\frac{1}{2}}. $$

Interpolating (\ref{l1bound}) and (\ref{l2bound}), we see that 
\begin{equation} \label{salemrestrictiontheoremestimate} {\left( \frac{1}{|\Sigma|} \sum_{m \in \Sigma} {|\widehat{f}(m)|}^2 \right)}^{\frac{1}{2}} \leq {\left( \Lambda_{\text{Salem}}^{\frac{1}{2}} \cdot {|\Sigma|}^{-\frac{1}{4}} \right)}^{1-\frac{2}{p'}} \cdot {\left( \frac{N^d}{|\Sigma|} \right)}^{\frac{1}{p'}} \cdot N^{-d} {\left( \sum_x {|f(x)|}^p \right)}^{\frac{1}{p}} \end{equation} 
$$ =\Lambda_{\text{Salem}}^{\frac{1}{2}-\frac{1}{p'}} {|\Sigma|}^{-\frac{1}{4}-\frac{1}{2p'}} N^{\frac{d}{p'}} \cdot N^{-d} \cdot {\left( \sum_x {|f(x)|}^p \right)}^{\frac{1}{p}}$$ for $1 \leq p \leq 2$. 

\vskip.125in 

In particular, if $|\Sigma| \approx N^{\alpha}$, $0<\alpha<d$ then under the assumptions of Theorem \ref{theoremsalem}, we obtain a $(p,2)$ restriction theorem for $p' \ge \frac{2(2d-\alpha)}{\alpha}$. See, for example, \cite{MT04} for analogous results. 

\end{remark}


\vskip.125in 
 
Theorem \ref{theoremsalem} leads to the natural question of which sets $S$ are Salem at level $\Lambda_{\text{Salem}}$, and this leads us to consider the following. Given $A \subset {\mathbb Z}_N^d$, define 
$$ \Phi(A)=\max \left\{|\widehat{A}(m)|: m \in {\mathbb Z}_N^d; \ m \not=\vec{0} \right\}.$$

Given that $|A|\leq \frac{N^d}{2}$, the quantity   $\Phi(A)$ is bounded from below and above as 
$$ N^{-d} \sqrt{\frac{|A|}{2}} \leq \Phi(A) \leq N^{-d}|A|,$$ where the upper bound follows from direct domination and the lower bound follows from Plancherel theorem and the assumption on the size of $A$. (For a detailed proof,  we refer to Proposition 2.6 in \cite{Babai2002}.) The following result addresses the question we raised above about when we can expect a Salem type estimate to hold. 

\begin{proposition}[\cite{Babai2002}, Proposition 5.2] \label{randomisgood} Let $\epsilon>0$. For all but $O(N^{-d\epsilon})$ subsets $A$ of ${\mathbb Z}_N^d$ of  size  $|A| \leq \frac{N^d}{2}$
\begin{equation} \label{nirvana} \Phi(A)< N^{-d} \sqrt{(1+\epsilon)|A| \cdot d \cdot \log(N)}. \end{equation} 
\end{proposition} 

\begin{remark} The proof of Proposition \ref{randomisgood} shows that if $A$ of a given size $\leq \frac{N^d}{2}$ is chosen randomly, with respect to the uniform probability distribution on ${\mathbb Z}_N^d$, then for any $\epsilon>0$, (\ref{nirvana}) holds with probability $1-N^{-d \epsilon}$. 

Also, observe that the random subset $S$ is significantly smaller than the total size of ${\mathbb Z}_N^d$, specifically less than half of it, chosen uniformly at random.
\end{remark}

\vskip.125in

\vskip.125in 
 
\subsection{The relationships between the parameters $\Lambda_{\text{size}}$, $\Lambda_{\text{Salem}}$ and $\Lambda_{\text{energy}}$} 

We are now going to exhibit some interesting relationships between the parameters we have been repeatedly using. 

\vskip.125in 

\begin{itemize}
\item[i)] ($\Lambda_{\text{Salem}}$ bound) As we noted above, the inequality 
$$ |\widehat{S}(z)| \leq \Lambda_{\text{Salem}} N^{-d} {|S|}^{\frac{1}{2}}$$ always holds for any set $S$ with $\Lambda_{\text{Salem}}={|S|}^{\frac{1}{2}}$. 

\vskip.125in 

\item[ii)] ($\Lambda_{\text{energy}}$ bound) The inequality 
$$ |\{(x,y,x',y') \in S^4: x+y=x'+y'\}| \leq \Lambda_{\text{energy}} {|S|}^2$$ always holds with $\Lambda_{\text{energy}}=|S|$, since we can fix $x,y,x'$ and solve for $y'$. 

\vskip.125in 

\item[iii)] (Random $\Lambda_{\text{Salem}}$ bound) It is also important to note that Proposition \ref{randomisgood} implies that if $S$ is chosen randomly and $|S|<\frac{N^d}{2}$, then with probability $1-N^{-d \epsilon}$ we may take $\Lambda_{\text{Salem}} \leq \sqrt{(1+\epsilon) \log(N)}$.

\vskip.125in 

\item[iv)] ($\Lambda_{\text{Salem}}$ versus $\Lambda_{\text{energy}}$) 
By a simple calculation, we have  
 \begin{align}\notag \sum_z {|\widehat{S}(z)|}^4&=N^{4d} \sum_{x,y,x',y'} \chi(z \cdot (x+y-x'-y')) S(x)S(y)S(x')S(y')\\\notag
&=N^{-3d} |\{(x,y,x',y') \in S^4: x+y=x'+y'\}|, 
\end{align} 
i.e., 
\begin{align}\label{energy} |\{(x,y,x',y') \in S^4: x+y=x'+y'\}|=N^{3d} \sum_z {|\widehat{S}(z)|}^4. 
\end{align}
 
 Suppose that $S$ satisfies $|\widehat{S}(z)| \leq \Lambda_{\text{Salem}} N^{-d} \cdot {|S|}^{\frac{1}{2}}$ for $z \not=(0, \dots, 0)$. By this assumption, 
 the right-hand side of \eqref{energy} is bounded by 
$$ N^{3d} \cdot \Lambda_{\text{Salem}}^2 \cdot N^{-2d} \cdot |S| \cdot \sum_z {|\widehat{S}(m)|}^2.$$ 

By Plancherel, this expression equals 
$$ \Lambda_{\text{Salem}}^2 \cdot {|S|}^2,$$ from which we conclude that 
\begin{equation} \label{energyvsfourier} \Lambda_{\text{energy}} \leq \Lambda_{\text{Salem}}^2. \end{equation} 

\vskip.125in 

\item[v)] The calculation above shows that a good Fourier bound (small $\Lambda_{\text{Salem}}$) leads to a good energy bound (small $\Lambda_{\text{energy}}$). We are about to see that the converse is much more problematic. 

\vskip.125in 

Here is a sketch of an example in the prime setting, but similar examples can be constructed for any $N$. Let $N$ be a large prime number, and let $E$ denote the disjoint union of the parabola 
$\{x \in {\mathbb Z}_p^2: x_2=x_1^2\}$ and an arithmetic progression on a line of length $\approx p^{\alpha}$, with $0<\alpha<\frac{2}{3}$. A direct calculation shows that 
$$ |\{(a,b,c,d) \in E^4: a+b=c+d \}| \approx {|E|}^2$$ because rich additive properties of the arithmetic progression on a line do not interfere the poor additive properties of the parabola because the arithmetic progression is too small. Please note that this calculation requires the primality of $N$. 

Now, $\widehat{E}(m)=\widehat{S}(m)+\widehat{L}(m)$, where $S$ is the indicator function of the parabola, and $L$ is the indicator function of a line. By classical Gauss sum estimates (see e.g. \cite{IR07}), 
$$ |\widehat{S}(m)| \leq N^{-\frac{3}{2}},$$ and this estimate is exact for all $m$ with $m_2 \not=0$. On the other hand, it is not difficult to find $m$ with $m_2 \not=0$ such that $|\widehat{L}(m)| \approx |L|p^{-2}=p^{\alpha-2}$. It follows that 
$$ |\widehat{E}(m)| \approx p^{-\frac{3}{2}}+p^{\alpha-2} \approx p^{\alpha-2}=p^{-\frac{3}{2}} \cdot p^{\alpha-\frac{1}{2}},$$ so 
$$ \Lambda_{\text{Salem}}=p^{\alpha-\frac{1}{2}}$$
as long as $\alpha>\frac{1}{2}$. On the other hand, $\Lambda_{\text{energy}} \approx 1$ in this case. This shows that the inequality (\ref{energyvsfourier}) can be very far from equality. 

\vskip.125in 

\item[vi)] In view of Remark \ref{randomenergytheorem}, in the case when $S$ is chosen randomly, $\Lambda_{\text{energy}} \leq \Lambda_{\text{size}}^2$. 

\end{itemize}

\section{Signal recovery in $\Bbb Z_N^d$}
\label{ersection}

\subsection{Exact recovery via restriction theory} We begin with the exact recovery mechanism that follows from Theorem \ref{mainUPwithRT}. 

\begin{corollary}[Exact Recovery via Restriction Estimation]\label{ERwithRT} Let $f:\Bbb Z_N^d\to \Bbb C$ be a signal supported in $E \subset {\mathbb Z}_N^d$. Let $r$ be a frequency bandlimited signal 
such that 
$$
  \widehat{r}(m) =
\begin{cases}
    \widehat{f}(m), & \text{for } m \notin S \\
    0, & \text{otherwise}. 
\end{cases}
$$
 
 Suppose that (\ref{restrictionequation}) holds for $S$,  the set of unobserved frequencies of $f$. Then $f$ can be reconstructed from $r$  uniquely if 
\begin{align}\label{ineq:x} {|E|}^{\frac{1}{p}} \cdot |S| < \frac{N^d}{2^{\frac{1}{p}} C_{p,q}}. 
\end{align}
\end{corollary} 

\vskip.125in 

This result follows from Theorem \ref{mainUPwithRT} using the Donoho-Stark mechanism described in Subsection \ref{subsubsectionuptoer}. 

\vskip.125in 

\begin{remark} In view of Remark \ref{randomenergytheorem}, we can replace the assumption on $S$ in Corollary \ref{ERwithRT} with the assumption that $S$ is randomly selected with uniform probability and satisfies the size condition \eqref{ineq:x}. 
\end{remark} 

\vskip.125in 

\subsection{DRA algorithm and recovery of  0-1 signals via restriction theory} \label{simplerecoverysubsection}
\footnote{Note that there are many reasons why someone may want to transmit a higher dimensional signal. For example, a graph on $N$ vertices can be specified via its adjacency matrix, which is an $N$ by $N$ matrix of $1$s and $0$s, which can be encoded as an indicator function of a subset of ${\mathbb Z}_N^2$ corresponding to the $1$ entries in the matrix.} Donoho and Stark \cite{DS89} provide an algorithm for reconstructing the signal $f$, with a certain degree of complexity. Our observation is that in the case of $0-1$ signals, the recovery mechanism is very simple and via  Direct Rounding Algorithm under more ``sparsity" conditions, as we illustrate in the next theorem. 

\vskip.12in 

\begin{definition}[Direct Rounding Algorithm (DRA)] \label{dradefinition} Let $E, S\subset {\mathbb Z}_N^d$ and let $E(x)$ denote its indicator function. Suppose that the values of $\widehat{E}(m)$ are not known for $m \in S$. Let $r$ be a frequency bandlimited signal obtained by  a sharp frequency ``cut-off" map $P_B$: 
$$ r := P_B(E),$$
Then $$r(x) = \sum_{m\in B} \hat E(m)  \chi(m\cdot x),$$
and  $\widehat{r}(m)=\widehat{E}(m)$ for $m \notin S$, and $0$ otherwise.

\vskip.125in 

Let $G(x)$ be defined as follows. If $|r(x)| \ge .5$, then $G(x)=1$, otherwise $G(x)=0$. We say that $E$ can be recovered via the {\it Direct Rounding Algorithm}   if $E(x)=G(x)$ for all $x \in {\mathbb Z}_N^d$. 
\end{definition} 

\begin{figure}
\label{IDFT-signelton}
\centering
d\includegraphics[scale=.5]{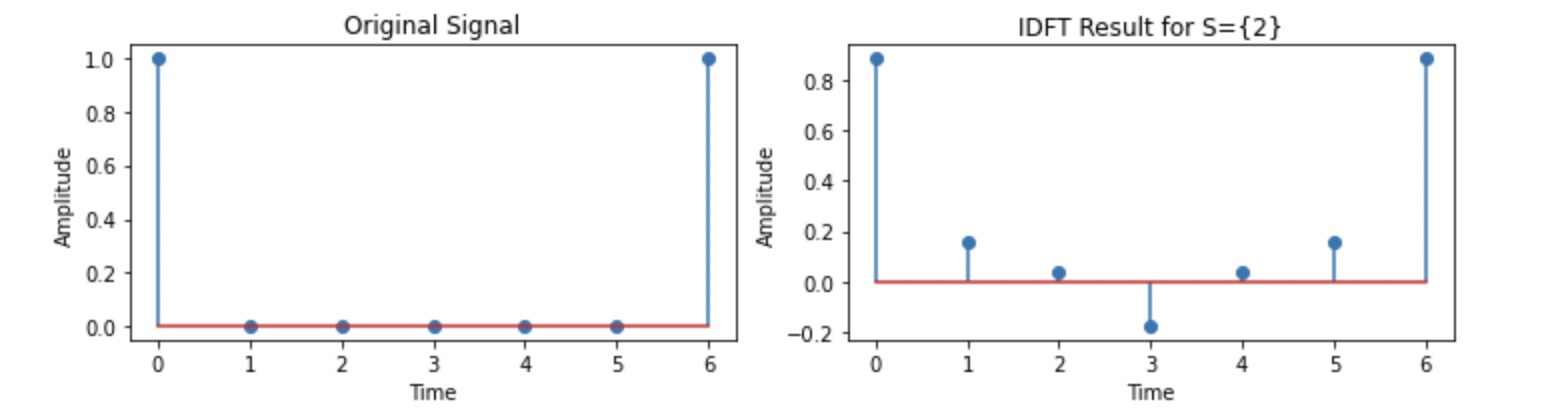}
 \caption{\tiny{Left: Graph of original 1-bit signal.   Right: Inverse discrete Fourier transform of the signal in the presence of a single missing Fourier measurement $S$. The recovery of original signal is obtained from DRA (or threshelding) of  IDFT.}} 
\end{figure}

 The next results illustrate that the Direct Rounding Algorithm  can be   effectively applied  for the recovery of  binary signals under some specific  size conditions for the sets $E$ and $S$. 

\begin{theorem}\label{DRAtheorem} Let $E$ be a binary signal in ${\mathbb Z}_N^d$. 
\begin{itemize}
\item[i)] Suppose that the frequencies in $S \subset {\mathbb Z}_N^d$ are unobserved. Then $E$ can be recovered via DRA provided that
\begin{equation} \label{simplerecoveryequation} |E| \cdot |S|<\frac{N^d}{4}, \end{equation}
holds. 
\item[ii)] Suppose that the frequencies in $S \subset {\mathbb Z}_N^d$ are unobserved and $S$ satisfies the restriction estimate (\ref{restrictionequation}), then $E$ can be recovered via DRA provided that  
\begin{equation} \label{RTDRAequation} {|E|}^{\frac{1}{p}} \cdot |S|<\frac{N^d}{2C_{p,q}},\end{equation} 
holds. 
\end{itemize}
\end{theorem} 
\begin{proof}
 (i) 
Let $E \subset {\mathbb Z}_N^d$ and let $E(x)$ denote its indicator function, i.e., $E(x)=1$ when $x\in E$ and $E(x)=0$ otherwise. Suppose that  $S \subset {\mathbb Z}_N^d$. We can write 
\begin{align} E(x)&=\sum_{m \in {\mathbb Z}_N^d} \chi(x \cdot m) \widehat{E}(m) \\\notag
 &=\sum_{m \notin S} \chi(x \cdot m) \widehat{E}(m)+\sum_{m \in S} \chi(x \cdot m) \widehat{E}(m)=I(x)+II(x).\end{align}

Suppose that the frequencies in $S$ are unobserved. Under the assumption \eqref{simplerecoveryequation}, the signal $E$ can be recovered directly via DRA. Indeed, by the Cauchy-Schwarz inequality, we estimate the error term from above: 
\begin{align}
    \label{DRAequation} |II(x)| & \leq {|S|}^{\frac{1}{2}} \cdot {\left( \sum_{m \in S} {|\widehat{E}(m)|}^2 \right)}^{\frac{1}{2}}  \\  
 & \leq {|S|}^{\frac{1}{2}} \cdot {\left( \sum_{m \in {\mathbb Z}_N^d} {|\widehat{E}(m)|}^2 \right)}^{\frac{1}{2}} =N^{-\frac{d}{2}} {|S|}^{\frac{1}{2}} \cdot {|E|}^{\frac{1}{2}}. 
\end{align}  

 Notice that by the assumption \eqref{simplerecoveryequation} we have
 
\begin{equation} \label{simplerecoveryequation} |II(x)|<\frac{1}{2}. 
\end{equation}
Now, by applying DRA to  $r(x)= E(x)-I(x)$,  we can successfully recover the entire singal $E$. 

\vskip.12in 

We note that the DRA algorithm, described above, is executed as follows in this context. We take $I(x)$, compute its complex modulus, then round up to $1$ if $|E(x)-I(x)| \ge 0.5$, and round it down to $0$ otherwise. This is because $E(x)$ is equal to $1$ or $0$ and the error of $<\frac{1}{2}$ does not interfere with the rounding process. 





\vskip.125in 

(ii)
With the assumption that the restriction estimation (\ref{restrictionequation}) holds for  $S$, we proceed to adapt our previous argument as follows.  We have 
$$ |II(x)| \leq {|S|}^{\frac{1}{q'}} \cdot {\left( \sum_{m \in S} {|\widehat{E}(m)|}^q \right)}^{\frac{1}{q}} =|S| \cdot {\left( \frac{1}{|S|} \sum_{m \in S} {|\widehat{E}(m)|}^q \right)}^{\frac{1}{q}}$$
$$ \leq C_{p,q} N^{-d} \cdot |S| \cdot {\left( \sum_{x \in {\mathbb Z}_N^d} {|E(x)|}^p \right)}^{\frac{1}{p}}=C_{p,q} N^{-d} \cdot |S| \cdot {|E|}^{\frac{1}{p}}.$$ 

We conclude that exact recovery via DRA is possible for $0-1$ signals under the assumption  (\ref{restrictionequation}), provided that 
\begin{equation} \label{RTDRAequation} {|E|}^{\frac{1}{p}} \cdot |S|<\frac{N^d}{2C_{p,q}},\end{equation} a slightly more stringent condition than the one in Corollary \ref{ERwithRT}. 
 \end{proof}

 \begin{remark}
     In Theorem \ref{DRAtheorem}  (i),  we achieve a very simple exact recovery process. The price that we pay for this simple algorithm is that (\ref{simplerecoveryequation}) is more restrictive, by a factor of $\frac{1}{2}$, compared to the condition $|E| \cdot |S| <\frac{N^d}{2}$ that arises when we prove the exact recovery directly using the uncertainty principle in (\ref{UPdequation}). The same 
     argument holds true for 
     (ii).  
 \end{remark}

\vskip.125in 

\begin{remark} It is interesting to note that if $f: {\mathbb Z}_N^d$ has a bounded range and takes only a finite number of values, then the DRA mechanism can be applied, much like above, except that we need to bound $|II(x)|$ by $\frac{1}{2k}$ instead of $\frac{1}{2}$. 
\end{remark}

\subsection{Signal recovery via the Salem uncertainty principle } \label{simplerecoveryfouriersubsection}
\label{sizebinarysubsection} We are now going to explore the exact recovery consequences of the Salem Uncertainty Principle (Theorem \ref{theoremsalem}). Our main result in this direction is the following. 

\begin{theorem} \label{mainsalembasic} Let $f:\Bbb Z_N^d\to \Bbb C$ be a signal supported in $E \subset {\mathbb Z}_N^d$. Let $r$ be a frequency bandlimited signal obtained by  a sharp frequency ``cut-off" map $P_B$: 
$$ r := P_B(f),$$
where $P_B=\mathcal F^{-1} \chi_B\mathcal F$ and $B=\Bbb Z_N^d\setminus S$. Then $\widehat{r}(m)=\widehat{f}(m)$ for $m \notin S$, and $0$ otherwise. Suppose that $S$ is Salem at level $\Lambda_{\text{Salem}}$. Then $f$ can be reconstructed from $r$  uniquely if 
\begin{equation} \label{salemconditionformula} |E| \cdot {|S|}^{\frac{3}{4}} < \frac{1}{2} \cdot N^d \cdot \sqrt{\frac{1-dens(S)}{\Lambda_{\text{Salem}}}}.\end{equation}

\end{theorem} 

\vskip.125in 

This result follows from Theorem \ref{theoremsalem} using the Donoho-Stark mechanism described in Subsection \ref{subsubsectionuptoer}. 

\vskip.125in 

\begin{remark} In view of Proposition \ref{randomisgood}, we can replace the assumption on $S$ in Theorem \ref{mainsalembasic} by the assumption that $S$ is chosen randomly with respect to uniform probability. Then the conclusion that $f$ can be reconstructed from $r$  uniquely if (\ref{salemconditionformula}) holds with $\Lambda_{\text{Salem}}=\log((1+\epsilon) \cdot d \cdot N)$ is valid with probability $1-N^{-d \epsilon}$. 
\end{remark} 

\vskip.125in 

In the realm of $0-1$ signals, we can use Theorem \ref{theoremsalem} and run the DRA mechanism from Subsection \ref{simplerecoverysubsection} to obtain the following result. 

\begin{theorem} \label{salemdratheorem} Let $E \subset {\mathbb Z}_N^d$ and identify $E$ with its indicator function. Let $r$ be a frequency bandlimited signal obtained by a sharp frequency ``cut-off" map $P_B$: 
$$ r := P_B(E),$$
where $P_B=\mathcal F^{-1} \chi_B\mathcal F$ and $B=\Bbb Z_N^d\setminus S$. Then $\widehat{r}(m)=\widehat{E}(m)$ for $m \notin S$, and $0$ otherwise. Suppose that $S$ is Salem at level $\Lambda_{\text{Salem}}$. Then $E$ can be reconstructed from $r$ uniquely via the DRA (Direct Rounding Algorithm) if 
$$ |E| \cdot {|S|}^{\frac{3}{4}} < \frac{1}{2} \cdot N^d \cdot \sqrt{\frac{1-dens(S)}{\Lambda_{\text{Salem}}}}.$$ 
\end{theorem}

\section{Signal recovery in $\Bbb R^d$} \label{grsection} The signal recovery problem can be set up in a very general setting, such as manifolds, hyperbolic domains, fractals, and Lie groups. We shall address this issue in the sequel, but in the meantime, we are going to provide a simple illustration of how the concepts of this paper play out in the context of the celebrated restriction conjecture in ${\mathbb R}^d$, $d \ge 2$. 


In Euclidean spaces, we may consider the following version of the exact recovery problem. Let $A$ be a subset of the unit cube, say, of positive Lebesgue measure, and let $1_A(x)$ denote its indicator function. By  the inverse Fourier transform,  
$$ 1_A(x)=\int e^{2 \pi i x \cdot \xi} \ \widehat{1}_A(\xi) d\xi,  \quad \forall \ x\in \Bbb R^d.$$
  
Suppose that the values of $\widehat{1}_A(\xi)$ for  $\xi\in S^{\delta}$ are missing, where $S^{\delta}$ is the $\delta$-neighborhood of $S  \subset {\mathbb R}^d$. 

As before, we have 
$$ 1_A(x)=\int_{\xi \notin S^{\delta}} e^{2 \pi i x \cdot \xi} \ \widehat{1}_A(\xi) d\xi+\int_{S^{\delta}} e^{2 \pi i x \cdot \xi} \ \widehat{1}_A(\xi) d\xi=I+II,$$ where  for some $r\in [1,\infty)$
 
  \begin{align}\label{ineq:Y}
  | II | \leq |S^{\delta}| \cdot {\left( \frac{1}{|S^{\delta}|} \int_{S^{\delta}} {|\widehat{1}_A(\xi)|}^r d\xi \right)}^{\frac{1}{r}}. 
  \end{align}

\begin{definition} (Restriction in ${\mathbb R}^d$) Given a set $S \subset {\mathbb R}^d$, and a measure $\sigma_S$ supported on $S$, we say that a $(p,r)$ restriction theorem holds for $S$ if  for any function $f$ 
$$ {\left( \int_S {|\widehat{f}(\xi)|}^r d\sigma_S(\xi) \right)}^{\frac{1}{r}} \leq C_{p,r} {\left( \int_{\Bbb R^d} {|f(x)|}^p dx \right)}^{\frac{1}{p}}.$$
\end{definition} 

Suppose that $S$ is compact. For any $\delta>0$, let $S^\delta$ denote a $\delta$ neighborhood of $S$. Define 
\begin{equation} \label{thickeningmeasure} \sigma_S = \lim_{\delta\to 0^+} \frac{1}{|S^\delta|} 1_{S^\delta}.\end{equation} For example, if $S$ is the unit sphere, (\ref{thickeningmeasure}) is a natural way to define the classical surface measure. 

Moving right along, if a $(p,r)$- restriction theorem is valid for $S^{\delta}$, with constants independent of $\delta$ (if $\delta$ is sufficiently small), the expression on the right of \eqref{ineq:Y} above is bounded by  

$$ C_{p,r}  |S^{\delta}| \cdot {|A|}^{\frac{1}{p}}.$$ 

Suppose, for example, that $S$ has upper Minkowiski dimension $\alpha$. Then we conclude that 
\begin{equation} \label{euclideanpushpull} | II | \leq C_{p,r} \cdot \delta^{d-\alpha} \cdot {|A|}^{\frac{1}{p}}. \end{equation} 

The restriction theorem always holds with $p=1$, so we always have 

\begin{equation} \label{euclideantrivial} | II | \lesssim \delta^{d-\alpha} |A|, \end{equation} and exact recovery is possible if $\delta^{d-\alpha} |A|$ is smaller than a sufficiently small constant. If $S$ is a compact piece of a hyperplane, for example, then it is not difficult to see that we can never obtain a $(p,r)$ restriction estimate with $p>1$. However, we can say much more in some specific cases, like the cases of a sphere or a paraboloid due to their curvature properties. See, for example, the discussion of restriction theory in \cite{St93}. See also \cite{Mockenhaupt2000} for the discussion of restriction for sets of fractional dimension. 

\vskip.125in 

\begin{conjecture} \label{restrictionconjecture} (Restriction conjecture) The restriction conjecture says that if $S$ is the unit sphere, (see e.g. \cite{St93};  for a thorough description of the problem, and \cite{Wang2022} for some recent developments) then
\begin{equation} \label{continuousrestriction} {\left( \int_S {|\widehat{f}(\xi)|}^r d\sigma_S(\xi) \right)}^{\frac{1}{r}} \leq C_{p,r} {\left( \int_{{\Bbb R}^d} {|f(x)|}^p dx \right)}^{\frac{1}{p}} \end{equation} whenever 
$$ p<\frac{2d}{d+1}, \ r \leq \frac{d-1}{d+1}p',$$ where $p'$ is the conjugate exponent to $p$. 
\end{conjecture} 

\vskip.125in 

\begin{remark} In every known result pertaining to the restriction conjecture, the resulting estimate is still valid if $\sigma_S$ is replaced by $\frac{1}{|S^\delta|} 1_{S^\delta}$, with constants independent of $\delta$, if $\delta$ is sufficiently small. 
\end{remark} 

\vskip.125in 

\begin{theorem} \label{continuoussignalrecoverytheorem} Suppose that the restriction conjecture (\ref{continuousrestriction}) holds. Suppose that the same estimate holds if $\sigma_S$ is replaced by $\frac{1}{|S^\delta|} 1_{S^\delta}$ with $\delta$ sufficiently small. Let $A$ be a measurable subset of ${\Bbb R}^d$ and the Fourier transform of $\widehat{1}_A(\xi)$ is known, except for the $\delta$-neighborhood of the unit sphere. Then there exists $C<\infty$, independent of $\delta$, such that exact recovery of $A$ is possible, up to a set of measure $0$, if 
$$ |A| \leq C \delta^{-p} \ \text{for any} \ p<\frac{2d}{d+1}.$$ 
\end{theorem} 

\vskip.125in 

The proof of Theorem \ref{continuoussignalrecoverytheorem} follows by taking $\alpha=d-1$ and $p$ from the restriction conjecture in (\ref{continuousrestriction}) above. 



\section{Proof of Theorems}
\label{proofsection}

\begin{proof}[Proof of Theorem \ref{mainUPwithRT}]
Suppose that $f$ is supported in a set $E$, and $\widehat{f}$ is supported in a set $\Sigma$. Then by the Fourier Inversion Formula and the support condition, 
$$ f(y)=\sum_{m \in {\mathbb Z}_N^d} \chi(y \cdot m) \widehat{f}(m)=\sum_{m \in \Sigma} \chi(y \cdot m) \widehat{f}(m)$$

By H\"older's inequality, 
$$ |f(y)| \leq |\Sigma| \cdot {\left( \frac{1}{|\Sigma|} \sum_{m \in \Sigma} {|\widehat{f}(m)|}^q \right)}^{\frac{1}{q}}.$$ 

By  restriction bound assumption  (\ref{restrictionequation}), this expression is bounded by 
$$ |\Sigma| \cdot C_{p,q} \cdot N^{-d} \cdot {\left( \sum_{x \in {\mathbb Z}_N^d} {|f(x)|}^p \right)}^{\frac{1}{p}},$$ and by the support assumption, this quantity is equal to  
$$ |\Sigma| \cdot C_{p,q} \cdot N^{-d} \cdot {\left( \sum_{x \in E} {|f(x)|}^p \right)}^{\frac{1}{p}}.$$

Putting everything together, we see that 
$$ |f(y)| \leq |\Sigma| \cdot C_{p,q} \cdot N^{-d} \cdot {\left( \sum_{x \in E} {|f(x)|}^p \right)}^{\frac{1}{p}}\quad \forall \ y\in E.$$

Raising both sides to the power of $p$, summing over $E$, and dividing both sides of the resulting inequality by $\sum_{x \in E} {|f(x)|}^p$, we obtain 
$$ {|\Sigma|}^p \cdot |E| \cdot C_{p,q}^p \ge N^{dp},$$ or, equivalently, 
$$ {|E|}^{\frac{1}{p}} \cdot |\Sigma| \ge \frac{N^d}{C_{p,q}}, $$ as desired.

 \end{proof}

 \begin{proof}[Proof of Theorem \ref{restrictionlambda4}]
We have 
\begin{align} \label{initiallambda} \sum_{m \in \Sigma} {|\widehat{f}(m)|}^2 &=\sum_{m\in \Bbb Z_N^d} {|\widehat{f}(m)|}^2 \Sigma(m)  \\ 
  \label{gfunction}
  &=\sum_{m\in \Bbb Z_N^d}  \widehat{f}(m)\Sigma(m)g(m), \end{align} where 
$$ g(m)=\overline{\widehat{f}(m)}\Sigma(m).$$ 

By definition of the Fourier transform, the right-hand side of (\ref{gfunction}) is equal to 
 \begin{align}\notag
     &N^{-d} \sum_m \sum_x \chi(-x \cdot m) f(x) \Sigma(m)g(m)\\
  \label{beforeholder}
  &=\sum_x f(x) \widehat{g\Sigma}(x).  
\end{align}
By H\"older's inequality, the quantity in  (\ref{beforeholder})   is bounded by 
\begin{equation} \label{afterholder}  {\left( \sum_{x \in {\mathbb Z}_N^d} {|f(x)|}^{\frac{4}{3}} \right)}^{\frac{3}{4}} \cdot {\left( \sum_{x \in {\mathbb Z}_N^d} {|\widehat{g\Sigma}(x)|}^4 \right)}^{\frac{1}{4}}. \end{equation} 

Continuing, we have 
\begin{align}\notag
&\sum_{x \in {\mathbb Z}_N^d} {|\widehat{g \Sigma}(x)|}^4 =  \\\notag
&=N^{-4d} \sum_x \sum_{m_1,m_2,m_3,m_4 \in \Sigma} \chi(x \cdot (m_1+m_2-m_3-m_4)) g(m_1)g(m_2)g(m_3)g(m_4)\\\notag
&=N^{-3d} \sum_{m_1+m_2=m_3+m_4;  \ m_j \in \Sigma} g(m_1)g(m_2)g(m_3)g(m_4).
\end{align}

The modulus of this expression is bounded by 
$$ \Lambda_{\text{energy}} \cdot N^{-3d} \cdot {\left( \sum_m {|g(m)|}^2 \right)}^2.$$  
To see this, we use a similar idea in  \cite{KP22}, page 11:  we take $g$ to be a linear combination of indicator functions of sets, then apply the Cauchy-Schwartz and the assumption (\ref{lambda4condition}). 

\vskip.125in 

Going back, we see that the expression is bounded by 
$$  {\left( \sum_{x \in {\mathbb Z}_N^d} {|f(x)|}^{\frac{4}{3}} \right)}^{\frac{3}{4}}  \cdot \Lambda_{\text{energy}}^{\frac{1}{4}} \cdot N^{-\frac{3d}{4}} \cdot  {\left( \sum_m {|g(m)|}^2 \right)}^{\frac{1}{2}}.$$ 

If we go back to (\ref{initiallambda}), and unravel the definitions, we see that 
\begin{align}\notag
 \sum_m {|g(m)|}^2 \leq {\left( \sum_{x \in {\mathbb Z}_N^d} {|f(x)|}^{\frac{4}{3}} \right)}^{\frac{3}{4}}  \cdot \Lambda_{\text{energy}}^{\frac{1}{4}} \cdot N^{-\frac{3d}{4}} \cdot  
{\left( \sum_m {|g(m)|}^2 \right)}^{\frac{1}{2}}, 
\end{align} 
hence 
\begin{align}\notag {\left( \frac{1}{|\Sigma|} \sum_{m \in \Sigma} {|\widehat{f}(m)|}^2 \right)}^{\frac{1}{2}} &\leq {\left( \sum_{x \in {\mathbb Z}_N^d} {|f(x)|}^{\frac{4}{3}} \right)}^{\frac{3}{4}}  \cdot 
\frac{1}{{|\Sigma|}^{\frac{1}{2}}} \cdot \Lambda_{\text{energy}}^{\frac{1}{4}} \cdot N^{-\frac{3d}{4}}\\\notag
& =\Lambda_{\text{energy}}^{\frac{1}{4}} \cdot N^{-d} \cdot {\left( \sum_{x \in {\mathbb Z}_N^d} {|f(x)|}^{\frac{4}{3}} \right)}^{\frac{3}{4}} \cdot \frac{N^{\frac{d}{4}}}{{|\Sigma|}^{\frac{1}{2}}}\\\notag
&=\Lambda_{\text{size}}^{-\frac{1}{2}} \cdot \Lambda_{\text{energy}}^{\frac{1}{4}} \cdot N^{-d} \cdot {\left( \sum_{x \in {\mathbb Z}_N^d} {|f(x)|}^{\frac{4}{3}} \right)}^{\frac{3}{4}}, 
\end{align}
as claimed. 
\end{proof}


\begin{proof}[Proof of Theorem \ref{theoremsalem}] 
By Fourier Inversion, 
$$ f(x)=\sum_{m \in S} \chi(x \cdot m) \widehat{f}(m).$$

It follows that 
$$ |f(x)| \leq |S| \cdot {\left( \frac{1}{|S|} \sum_{m \in S} {|\widehat{f}(m)|}^2 \right)}^{\frac{1}{2}}.$$

We have 
$$ \sum_{m \in S} {|\widehat{f}(m)|}^2=\sum_{m \in S} {|\widehat{f}(m)|}^2 S_0(m)+\frac{|S|}{N^d} \sum_{m \in S} {|\widehat{f}(m)|}^2,$$ where $S_0(m)=S(m)-\frac{|S|}{N^d}$. 

\vskip.125in 

It follows that 

 \begin{align}
     (1-dens(S)) \sum_{m \in S} {|\widehat{f}(m)|}^2 &=\sum_m {|\widehat{f}(m)|}^2 S_0(m)\\\notag
& =N^{-d} \sum_{x,y} \bar{f}(x) f(y) \widehat{S}_0(x-y) \leq N^{-d} \cdot \Lambda_{\text{Salem}} \cdot \frac{{|S|}^{\frac{1}{2}}}{N^d} \cdot {\left( \sum_x |f(x)| \right)}^2. 
\end{align}

We deduce that 
$$ {\left( \frac{1}{|S|} \sum_{m \in S} {|\widehat{f}(m)|}^2 \right)}^{\frac{1}{2}} \leq 
\frac{N^{-d} \cdot {|S|}^{-\frac{1}{4}} \cdot \Lambda^{\frac{1}{2}}_{\text{Salem}} \cdot 
\sum_x |f(x)|}{\sqrt{1-dens(S)}}.$$

\vskip.125in 

Putting everything together, we see that $$ |f(x)| \leq \frac{1}{N^d} \cdot {|S|}^{\frac{3}{4}} \cdot \Lambda^{\frac{1}{2}}_{\text{Salem}} \cdot \sum_x |f(x)| \cdot \frac{1}{\sqrt{1-dens(S)}}.$$

\vskip.125in 

Summing both sides over $x \in E$, using the assumption that $f$ is supported in $E$, and dividing both sides by $\sum_{x \in E} |f(x)|$, we obtain the conclusion of the theorem. This completes the proof. 
\end{proof}


\vskip.125in 
\begin{proof}[Proof of Theorem \ref{salemdratheorem}]
We have 
 \begin{align}
     E(x)&=\sum_m \chi(x \cdot m) \widehat{E}(m) \\\notag
  &=\sum_{m \notin S} \chi(x \cdot m) \widehat{E}(m)+ \sum_{m \in S} \chi(x \cdot m) \widehat{E}(m)\\ \notag
 &=I(x)+II(x). 
\end{align}

By Cauchy-Schwarz, 
$$ |II(x)| \leq |S| \cdot {\left( \frac{1}{|S|} \sum_{m \in S} {|\widehat{E}(m)|}^2 \right)}^{\frac{1}{2}}.$$

By the proof of Theorem \ref{theoremsalem} above, 
$$ |S| \cdot {\left( \frac{1}{|S|} \sum_{m \in S} {|\widehat{E}(m)|}^2 \right)}^{\frac{1}{2}}$$
$$ \leq \frac{1}{N^d} \cdot {|S|}^{\frac{3}{4}} \cdot \Lambda^{\frac{1}{2}}_{\text{Salem}} 
\cdot |E| \cdot \frac{1}{\sqrt{1-dens(S)}}.$$

We need this quantity to be $<\frac{1}{2}$ and the desired conclusion follows using the reasoning laid out in Subsection \ref{simplerecoverysubsection}.

\end{proof}



\end{document}